\documentclass[12pt]{amsart}

\usepackage{amssymb,amsmath,amsthm, url}
\usepackage{comment}
\usepackage{enumitem}
\usepackage[all]{xy}

\setenumerate{label=(\alph*), font=\normalfont}

\newtheorem{thm}[equation]{Theorem}
 \newtheorem{prop}[equation]{Proposition}
 \newtheorem{lem}[equation]{Lemma}
 \newtheorem{cor}[equation]{Corollary}

 \theoremstyle{definition}
 
 \newtheorem{defn}[equation]{Definition}
 \newtheorem{remark}[equation]{Remark}

\numberwithin{equation}{section}

\newcommand{\bbZ}{{\mathbb{Z}}}

\newcommand{\bbP}{{\mathbb{P}}}
\newcommand{\bbA}{{\mathbb{A}}}
\newcommand{\bbG}{{\mathbb{G}}}
\newcommand{\bbC}{{\mathbb{C}}}

\newcommand{\bbQ}{{\mathbb{Q}}}

\newcommand{\Span}{\operatorname{span}}
\newcommand{\Br}{\operatorname{Br}}

\newcommand{\Mat}{\operatorname{Mat}}
\newcommand{\GL}{\mathrm{GL}}

\newcommand{\PGL}{\mathrm{PGL}}

\newcommand{\Id}{\operatorname{Id}}

\newcommand{\Gr}{\operatorname{Gr}}

\newcommand{\id}{\mathrm{id}}

\newcommand{\Spec}{\operatorname{Spec}}

\newcommand{\Aut}{\operatorname{Aut}}

\newcommand{\Sym}{\operatorname{S}}

\newcommand{\Fields}{\mathbf{Fields}}

\newcommand{\Sets}{\mathbf{Sets}}
\newcommand{\moduli}[2]{\overline{M}_{#1,#2}}

\DeclareFontEncoding{OT2}{}{} 
\DeclareTextFontCommand{\textcyr}{\fontencoding{OT2}
    \fontfamily{wncyr}\fontseries{m}\fontshape{n}\selectfont}

\author{Mathieu Florence}
\address{Institut de Math\'ematiques de Jussieu, Universit\'e Paris 6,
place Jussieu, 75005 Paris, France}
\email{mathieu.florence@imj-prg.fr}
\thanks{Florence was partially supported by the French National Agency 
(Project GeoLie ANR-15-CE40-0012).}

\author{Zinovy Reichstein}
\address{Department of Mathematics\\University of British Columbia\\ BC, Canada V6T 1Z2}
\email{reichst@math.ubc.ca}
\thanks{Reichstein was partially supported by
National Sciences and Engineering Research Council of
Canada Discovery grant 253424-2017.}

\begin{document}

\title[Forms of moduli spaces]
{The rationality problem for forms of $\moduli{0}{n}$}

\keywords
{Rationality, moduli spaces of marked curves, twisted forms, Galois cohomology,
Brauer group}
\subjclass[2010]{14E08, 14H10, 20G15,  16K50}

\begin{abstract}
Let $X$ be a del Pezzo surface of degree $5$ defined over a field $F$.
A theorem of Yu.~I.~Manin and P.~Swinnerton-Dyer asserts that
every Del Pezzo surface of degree $5$ is rational. In this paper we 
generalize this result as follows. 
Recall that del Pezzo surfaces of degree $5$ over a field $F$
are precisely the $F$-forms of the moduli space $\moduli{0}{5}$
of stable curves of genus $0$ with $5$ marked points.
Suppose $n \geqslant 5$ is an integer,
and $F$ is an infinite field of characteristic $\neq 2$. 
It is easy to see that every twisted $F$-form of  
$\moduli{0}{n}$ is unirational over $F$.  We show that

\smallskip
(a) If $n$ is odd, then every twisted $F$-form of
$\moduli{0}{n}$ is rational over $F$.

\smallskip
(b) If $n$ is even, there exists a field extension $F/k$ and
a twisted $F$-form $X$ of $\moduli{0}{n}$ such that $X$
is not retract rational over $F$.
\end{abstract}

\maketitle

\section{Introduction}

Let $X$ be a del Pezzo surface of degree $5$ defined over a field $F$.
Yu.~I.~Manin~\cite[Theorem 3.15]{manin} showed that if $X$ has an $F$-point, 
then $X$ is rational over $F$.  P.~Swinnerton-Dyer~\cite{sd} then
proved that $X$ always has an $F$-point; for  alternative proofs
of this assertion, see~\cite{sb} and~\cite{skorobogatov}. 
In summary, one obtains the following result, published 
earlier by F.~Enriques~\cite{enriques} (with an incomplete proof). 

\begin{thm} \label{thm1} {\rm (}Enriques, Manin, Swinnerton-Dyer{\rm )}
Every del Pezzo surface of degree $5$ defined over a field $F$
is $F$-rational. Equivalently, every $F$-form of
$\moduli{0}{5}$ is $F$-rational.
\end{thm}

The purpose of this paper is to generalize this celebrated theorem
as follows.  
As usual, we will denote the 
moduli space of smooth (respectively, stable) curves of genus $g$
with $n$ marked points by $M_{g, n}$ (respectively, $\moduli{g}{n}$). 
Recall that these moduli spaces are defined
over the prime field.
A {\em form} of a scheme $X$ defined over a field $F$ is 
an $F$-scheme $Y$, such that $X$ and $Y$ become isomorphic
over the separable closure $F^{sep}$. We will use the terms
``form", ``$F$-form" and ``twisted form"
interchangeably throughout this paper. For a discussion 
of this notion and further references, see Section~\ref{sect.prel-moduli}.

We now recall that $\moduli{0}{5}$ is a split del Pezzo surface 
of degree $5$, and $F$-forms of $\moduli{0}{5}$ are precisely 
the del Pezzo surfaces of degree $5$ defined over $F$. 
The main result of this paper is Theorem~\ref{thm.main} below.

\begin{thm} \label{thm.main} Let $n \geqslant 5$ be an integer
and $F$ be an infinite field of characteristic $\neq 2$.

\smallskip
(a) Assume $n$ is odd. Then every $F$-form of $\moduli{0}{n}$
is rational over $F$.

\smallskip
(b) Assume $n$ is even. If $\Br_2(F) \neq 0$, then there exists 
an $F$-form $X$ of $\moduli{0}{n}$ such that $X$ is not retract 
rational over $F$.
\end{thm}

Several remarks are in order.

(1) If $F$ is assumed to be an infinite field of characteristic different from $2$, 
Theorem~\ref{thm1} is recovered from Theorem~\ref{thm.main}(a) by setting $n = 5$. 

(2) In part (b), $\Br_2$ denotes the $2$-torsion of the Brauer group. 
The condition that $\Br_2(F) \neq 0$ is fairly mild; 
it is equivalent to the existence
of a non-split quaternion algebra over $F$. In particular, 
$\Br_2(F) \neq 0$ if $F = k(t_1, t_2)$ 
where $k$ is an arbitrary field of characteristic $\neq 2$ 
and $t_1, t_2$ are independent variables. 

(3) The assumption that $F$ is infinite is only used in part (a); 
see Section~\ref{sect.main-a}. In part (b) it is automatic. Indeed, by a theorem of
J.~Wedderburn, $\Br_2(F) = (0)$ for any finite field $F$; see e.g.,~\cite[Remark 6.2.7]{gs}.

(4) For $n \geqslant 5$, all $F$-forms of $\moduli{0}{n}$ are unirational over $F$;
see~\cite[Theorem 6.1]{dr} or Proposition~\ref{prop.Noether}(a) below.

(5) It is natural to ask if similar rationality results hold for forms
of $\moduli{g}{n}$ for $g \geqslant 1$. Theorems of J.~Harris 
D.~Mumford, D.~Eisenbud and 
G.~Farkas~\cite{harris-mumford, eisenbud-harris,
farkas1, farkas2}, assert that
$M_{g,0}$ is not unirational for any $g \geqslant 23$, and
hence, neither is $M_{g,n}$ for any $n \geqslant 0$.  Moreover, 
A. Logan~\cite{logan} exhibited an explicit integer $f(g)$, for each
$1 \leqslant g \leqslant 22$ such that $M_{g, n}$ is 
not unirational as long as $n \geqslant f(g)$. Deciding
for which of the finitely many remaining pairs $(g, n)$ 
the moduli space $M_{g, n}$ is rational, stably rational or 
unirational, over $\mathbb C$ is a problem of ongoing interest; 
see, e.g.,~\cite{cf}. We have shown that in some cases 
(for small $n, g \geqslant 1$), every form of $\moduli{g}{n}$ 
is stably rational. We plan to publish these results in
a forthcoming paper~\cite{flor}.

The remainder of this paper will be devoted to proving Theorem~\ref{thm.main}.

\section{Preliminaries on moduli spaces of curves and their
twisted forms}
\label{sect.prel-moduli}

The $F$-forms of a quasi-projective variety
$X$ are in a natural bijective correspondence
with $H^1(F, \Aut(X))$; see~\cite[II.1.3]{serre-gc}.
Here $\Aut(X)$ is a functor which associates
to the scheme $S/F$ the abstract group $\Aut(X_S)$. This functor is
not representable by an algebraic group defined over $F$ in general.
If it is, one usually says that $\Aut(X)$ is an algebraic group.
In the case, where $\Aut(X)$ is an algebraic group,
the bijective correspondence between $H^1(F, \Aut(X))$ and the set
of $F$-forms of $X$ (up to $F$-isomorphism)
can be described explicitly by using the twisting operation.
That is, to an $\Aut(X)$-torsor $\tau \colon Y \to \Spec(F)$,
we associate the $F$-variety $^\tau X := (X \times Y)/\Aut(X)$,
which is a twisted form of $X$.
Up to $F$-isomorphism,  $^\tau X$ depends only on the class
$\alpha$ of $\tau$ in $H^1(F, \Aut(X))$; see~\cite[Section III.1.3]{serre-gc}.
By abuse of notation, we will sometimes write $^\alpha X$ 
in place of $^\tau X$.  For the definition and basic properties 
of the twisting operation
we refer the reader to~\cite[Section 2]{florence} or~\cite[Section 3]{dr}.
Conversely, to a twisted form
$X'$ of $X$ defined over $F$, we associate the $\Aut(X)$-torsor
$\operatorname{Isom}_{F}(X, X') \to \Spec(F)$. 

The following recent result is the starting point for our investigation.

\begin{thm} \label{thm.massarenti}
Let $F$ be a field of characteristic $\neq 2$.
If $2g + n \geqslant 5$, then the natural embedding
$\Sym_n \to \Aut_F(\moduli{g}{n})$ is an isomorphism.
\end{thm}

In the case $g = 0$ and $F = \bbC$, Theorem~\ref{thm.massarenti}
was proved by A.~Bruno and M.~Mella~\cite{bruno-mella}. In the more
general situation, where $F = \mathbb C$ but $g \geqslant 0$ is arbitrary,
it is due to A.~Massarenti~\cite{massarenti}, and in full generality 
to B.~Fantechi and A.~Massarenti~\cite[Theorem A.2 and Remark A.4]{fm}.
As an immediate consequence, we obtain the following.

\begin{cor} \label{cor.forms} Let $F$ be a field of characteristic $\neq 2$,
and $g, n$ be non-negative integers such that $2g + n \geqslant 5$.
Then every $F$-form of $\moduli{g}{n}$ is isomorphic to
${\,}^{\alpha} \moduli{g}{n}$ for some $\alpha \in H^1(F, \Sym_n)$.
\qed
\end{cor}

\begin{remark} 
If $\Sym_n$ is the full automorphism group of the moduli space $M_{0, n}$ 
of smooth marked curves, then $\moduli{0}{n}$ can be replaced 
by $M_{0, n}$ in the statement of Theorems~\ref{thm.main}.
The proof remains unchanged.
In particular, by~\cite[Section 4.10, Corollary 7]{lin}, this is the case
if $F = \mathbb C$. 
\end{remark}

\begin{remark} \label{rem.functors}
Recall that $\moduli{g}{n}$ is, by definition,
the coarse moduli space of the functor which assigns to a scheme
$X$, defined over $F$,
the set of isomorphism classes of pairs $(C, s)$,
where $C \to X$ is a stable curve of genus $g$ over $X$ and
$s = (s_1, \dots, s_n)$ is an $n$-tuple of disjoint sections
$s_i \colon X \to C$. Equivalently, we may view $s$ as a single
closed embedding $s \colon X^{(n)} \hookrightarrow  C$
(over $X$), where $X^{(n)} = X \times_{\Spec(F)} \Spec(F^n)$
is the disjoint union of $n$ copies of $X$.
To place our results into the context of moduli theory,
we remark that if $2g + n \geq 5$, then every form of $\moduli{g}{n}$
admits a similar functorial interpretation.
Suppose $\alpha \colon Y \to \Spec(F)$ is an $\Sym_n$-torsor 
represented by an $n$-dimensional \'etale algebra $E/F$.
Then ${\,}^{\alpha} \moduli{g}{n}$ is the coarse moduli space for the functor
\[X \mapsto \{\text{isomorphism classes of pairs $(C, s)$} \}, \]
where $C \to X$ is a stable curve of genus $g$, and
$s$ is an embedding $X \times_{\Spec(F)} \Spec(E) \to C$ (over $X$).
We will not use this functorial description of
${\,}^{\alpha} \moduli{g}{n}$ in the sequel.
\end{remark}

\section{Preliminaries on the Noether problem}

Let $G$ be a linear algebraic group, and
$G \to \GL(V)$ be a finite-dimensional
representation of $G$, both defined over a field $F$.
We will assume that this representation is generically free, i.e.,
there is a dense open subset $U \subset V$ such that the
scheme-theoretic stabilizer of every point of $U$ is trivial.

The following questions originated in the work of E.~Noether.
Here (R) stands for rationality, (SR) for stable rationality
and (RR) for retract rationality.

\smallskip
\noindent
Noether's problem (R): Is $F(V)^G$ rational over $F$?

\smallskip
\noindent
Noether's problem (SR): Is $F(V)^G$
stably rational over $F$? That is, is there a field $E/F(V)^G$
such that $E$ is rational over both $F(V)^G$ and $F$?

\smallskip
\noindent
Noether's Problem (RR):
Is $F(V)^G$ retract rational over $F$?

\smallskip
Recall that an irreducible variety $Y$ defined over $F$ is called
{\em retract rational} if the identity map $Y \to Y$ factors through
the affine space $\bbA_F^n$ for some $n \geqslant 1$:
\begin{equation} \label{e.retract}
\xymatrix{ Y  \ar@{->}[rr]^{\id} \ar@{-->}[dr]^i &  & Y\\
 &  \bbA_F^n \ar@{-->}[ur]^j . &   }
\end{equation}
Here $i$ and $j$ are composable rational maps,
i.e., the image of $i$ and the domain of $j$ intersect non-trivially.
A finitely generated field extension $L/F$ is called {\em retract rational} if
some (and thus any) model $Y$ of $L/F$ is retract rational. Here by a model
of $L/F$ we mean an irreducible variety $Y$ defined over $F$ such that
$F(Y) = L$.

Noether'r original paper~\cite{noether} only considered problem (R)
(and only in the case, were $G$ is a finite group and $V$ is
the regular representation of $G$). Subsequent
attempts to solve problem (R) naturally led to problems (SR) and (RR).
Note, in particular, that the answers to problems (SR) and (RR)
depend only on the group $G$ and not on the choice
of generically free representation $V$.  For this reason
we will refer to these problems as {\em Noether's problems (SR) and
(RR) for $G$} in the sequel.
The answer to problem (R) may a priori depend on the choice of $V$.

\begin{remark} \label{rem.special} {\rm (}see~\cite[Section~4.2]{cts}{\rm )} 
Suppose $G$ is a special group defined over $F$, i.e., 
$H^1(K, G) = \{ 1 \}$ for every field extension $K/F$. 
Recall that a special group is always linear and connected; see
\cite[Theorem 4.4.1.1]{serre-special}. 

Let $\pi \colon V \dasharrow V/G$ be the rational quotient map.
That is, $V/G$ is any variety defined over $F$ whose function field in 
$F(V)^G$, and $\pi$ is induced by the inclusion of fields 
$F(V)^G \hookrightarrow F(V)$.  If $G$ is special, $\pi$ has 
a rational section and thus
$V$ is birationally isomorphic to $V/G \times G$ over $F$.
Consequently, Noether's problem (SR) has a positive solution
for $G$ if and only if $G$ is itself stably rational over $F$, 
and similarly for Noether's problem (RR).
\end{remark}

\begin{defn} \label{def.r-trivial} We will say that
a $G$-torsor $\alpha$ over a field $K$ is $r$-trivial 
if it can be connected to the trivial
torsor by a rational curve.  In other words, $\alpha$ is $r$-trivial if 
there exists an open subset $C \subset \bbA^1$ defined 
over $K$, a $G$-torsor $Y \to C$, and $K$-points 
$p_1, p_2 \colon \Spec(K) \to C$ such that 
$p_1^*(Y) \simeq \alpha$ and $p_2^*(Y)$ is split.
\end{defn}

Note that our notion of 
$r$-triviality is a minor variant of the more commonly used notion
of $R$-triviality, introduced by Manin~\cite{manin2}.
A $G$-torsor $\alpha$ over $K$ is called 
$R$-trivial if it can be connected to the trivial torsor by
a chain of rational curves defined over $K$.

\begin{lem} \label{lem.R-connected}
Suppose Noether's problem (RR) has a positive solution for an
affine algebraic group $G/K$.  Then every $G$-torsor 
$\alpha \colon X \to \Spec(K)$ 
is $r$-trivial, for every infinite field $K$ containing $F$.
\end{lem}

\begin{proof} There is a dense $G$-invariant open 
subset $V_0 \subset V$ which is the total space 
of a $G$-torsor $\pi \colon V_0 \to Y$; see~\cite[Section 5]{gms}.  
Here $\pi^{\ast} F(Y) = F(V)^G$. 
Recall that we are assuming $Y$ is retract rational.
After replacing $Y$ by a dense open subset, we may further assume
that $Y$ is a locally closed subvariety of $\bbA^n$, 
$i \colon Y \dasharrow \bbA^n$ in~\eqref{e.retract} 
is the inclusion map, and $j \colon \bbA^n \dasharrow Y$ is regular on some
dense open subset $U$ of $\bbA^n$ containing $Y$.

It is well known that $\pi$ is a versal torsor; once again, 
see~\cite[Section 5]{gms} or~\cite{dr}.  In particular,
there is a $K$-point $p_1 \colon \Spec(K) \to Y$ such that 
$\pi$ restricts to $\alpha$ over $p_1$, i.e.,
$p_1^*(\pi) = \alpha$. Similarly, there is a point
$p_2 \colon \Spec(K) \to Y$ such that $\pi$ splits over $p_2$.
It now suffices to connect $p_1$ and $p_2$ by an affine
rational curve $C \subset Y$, defined over $K$, which is smooth
at $p_1$ and $p_2$. After removing a closed subset from $C$ away 
from $p_1$ and $p_2$, we may
assume that $C$ is isomorphic to an open subset of $\bbA_K^1$. Then
we obtain a torsor $T \to C$ with the desired properties
by pulling back $\pi$ to $C$. 

To construct $C$, we first connect $p_1$ and $p_2$ by 
a rational curve $C_0$ in $\bbA^n$, then set $C := j(C_0)$.
Note that since $j \colon U \to Y$ is the identity map on $Y$,
the differential $d j_{p}$ is surjective for every $p \in Y$.
Hence, we can choose $C_0$ so that $C$ is smooth at $p_1$ and $p_2$.
\end{proof}

\section{The Noether problem for a class of twisted groups}

Let $G_0 := G(F^n/F) = (\GL_2 \times \bbG_m^n)/\bbG_m$, where
$\bbG_m$ is centrally embedded into $\GL_2 \times \bbG_m^n$ by $t \mapsto
(t^{-1} \Id, t, \dots, t)$.  The group $G_0$ and its twisted forms,
\begin{equation} \label{e.G}
G(E/F) := (\GL_2 \times R_{E/F}(\bbG_m))/\bbG_m,
\end{equation}
where $E/F$ is an \'etale algebra of degree $n$,
will play a prominent role in the sequel.

Recall that $\moduli{0}{n}$ is $\Sym_n$-equivariantly
birationally isomorphic to $(\bbP^{1})^n/\PGL_2$. In turn, 
$(\bbP^{1})^n/\PGL_2$ is $\Sym_n$-equivariantly birationally 
isomorphic to \[ (\bbA^2)^n/(\GL_2 \times (\bbG_m)^n) \, . \]
Here we identify $\bbG_m^n$ with the diagonal maximal torus 
in $\GL_n$, and $(\bbA^2)^n$ with the affine space $\Mat_{2, n}$
of $2 \times n$ matrices.
The group $\GL_2$ acts on $\Mat_{2, n}$ via multiplication on 
the left, and the torus $\bbG_m^n$ acts via multiplication on the right.
These two commuting linear actions give rise to a linear representation 
\[ \GL_2 \times \bbG_m^n \to \GL(\Mat_{2, n}) \, . \]
One readily checks that the kernel of this representation is
\[ H = \{ (t^{-1} \Id, t, \dots, t) \in \GL_2 \times \bbG_m^n \, | \; 
t \in \bbG_m \} \simeq \bbG_m \]
and that the induced representation 
\[
\phi \colon G_0 = (\GL_2 \times \bbG_m^n)/\bbG_m \to \GL(\Mat_{2, n}) 
\]
is generically free (recall that we are assuming that $n \geqslant 5$
throughout).
Now identify $\Sym_n$ with the subgroup of permutation matrices 
in $\GL_n$, and let this group act on
$\Mat_{2, n}$ linearly, via multiplication on the right. 
In summary,
\begin{equation} \label{e.G_0}
\moduli{0}{n} \simeq (\bbP^{1})^n/\PGL_2 \simeq \Mat_{2, n}/G_0,
\end{equation}
where $\simeq$ denotes an $\Sym_n$-equivariant birational isomorphism.

Let $\tau$ be an $\Sym_n$-torsor over $\Spec(F)$.
Since $\Sym_n$ normalizes $\bbG_m^n$ in $\GL_n$, we can twist
the group $G_0$ and the representation
$\phi$ by $\tau$ and obtain a new group 
\begin{equation} \label{e.twistedG} 
{\, }^{\tau} G_0 := 
{\, }^{\tau} (\GL_2 \times R_{E/F}(\bbG_m))/ {\, }^{\tau} H := G(E/F)
\end{equation}
and a new representation ${ }^{\tau} \phi \colon
{ }^{\tau} G_0 \to \GL({ }^{\alpha} \Mat_{2, n})$ defined over $F$.
Note that $\Sym_n$ acts trivially on $H$, and thus
${\, }^{\tau} H \simeq H \simeq \bbG_m$ over $F$. Moreover,
by Hilbert's Theorem 90, 
${ }^{\tau} \Mat_{2, n}$ is isomorphic to $\Mat_{2, n}$ as 
an $F$-vector space. Explicitly,  
\[ { }^{\tau} G_0 \simeq G(E/F) := (\GL_2 \times R_{E/F}(\bbG_m))/\bbG_m \, , \]
${ }^{\tau} \Mat_{2, n}$ is the affine space $\bbA(F^2 \otimes_F E)$, where
$\GL_2$ acts $F$-linearly on $F^2 \otimes_F E$ via multiplication on $F^2$
and $R_{E/F}(\bbG_m))$ acts via multiplication on $E$.
We have thus proved the following:

\begin{prop} \label{prop.Noether} 
Let $F$ be a field, $\tau$ be an $\Sym_n$-torsor over $\Spec(F)$,
and $E/F$ be the \'etale algebra associated to $\tau$.

\smallskip
(a) {\rm (cf.~\cite[Theorem 6.1]{dr})} 
${ }^{\tau} \moduli{0}{n}$ is unirational.

\smallskip
(b) ${ }^{\tau} \moduli{0}{n}$ is rational over $F$ if and only if
Noether's problem (R) for the representation ${ }^\tau \phi$ 
of the group $G(E/F)$ has a positive solution.

\smallskip
(c) ${ }^{\tau} \moduli{0}{n}$ is stably rational over $F$ if 
and only if Noether's problem (SR) for the 
group $G(E/F)$ has a positive solution.

\smallskip
(d) ${ }^{\tau} \moduli{0}{n}$ is retract rational over $F$ 
if and only if Noether's problem (RR) for the 
group $G(E/F)$ has a positive solution.
\qed
\end{prop}

\section{The Galois cohomology of $G(E/F)$}

Let $E/F$ be a finite-dimensional \'etale algebra and
$G := G(E/F) := (\GL_2 \times R_{E/F}(\bbG_m))/\bbG_m$
be the algebraic group we considered in the previous section; see~\eqref{e.G}.

\begin{lem} \label{lem.Galois-cohomology}
Let $\mathcal{F} \colon \Fields_F \to \Sets$ be the functor from 
the category of field extensions of $F$ to the category of sets,  
defined as follows: 

\begin{eqnarray*} 
\text{$\mathcal{F}(K) := \{$isomorphism classes of quaternion $K$-algebras $A$} \\ 
\text{such that $A$ is split by $E \otimes_F K \}$.} 
\end{eqnarray*}

\smallskip
\noindent
Then the functors $\mathcal{F}$ and $H^1(\ast, G)$ are isomorphic.
\end{lem}

\begin{proof}  Consider the the short exact sequence  
\begin{equation} \label{e.ses}
1 \to R_{E/F}(\bbG_m) \to G \to \PGL_2 \to 1 
\end{equation}
of algebraic groups and the associated long exact sequence
\begin{equation} \label{e.gc}
 H^1(K, R_{E/F}(\bbG_m)) \to H^1(K, G)  \xrightarrow{\alpha} H^1(K, \PGL_2) \xrightarrow{\delta}
H^2(K, R_{E/F}(\bbG_m)) 
\end{equation}
of Galois cohomology sets.  By Shapiro's Lemma, 
\[ H^1(K, R_{E/F}(\bbG_m)) \simeq H^1(K \otimes_F E, \bbG_m) = \{ 1 \} \, , \]
and $H^2(K, R_{E/F}(\bbG_m)) \simeq H^2(K \otimes_F E, \bbG_m)$ is in
a natural bijective correspondence with the Brauer group $\Br(K \otimes_F E)$.
Thus the long exact sequence~\eqref{e.gc} simplifies to
\begin{equation} \label{e.gc2}
\{ 1 \} \to H^1(K, G)  \xrightarrow{\alpha} H^1(K, \PGL_2) 
\xrightarrow{\delta} \Br(K \otimes_F E) \, . 
\end{equation}
Here $H^1(K, \PGL_2)$ is the set of isomorphism classes
of quaternion algebras $A/K$. 
The connecting map $\delta$ takes an algebra $A/K$ to 
$A \otimes_K (K \otimes_ F E)$.  By~\cite[Proposition 42]{serre-gc},  
$\alpha$ 
is injective.\footnote{Note that a priori the exact sequence~\eqref{e.gc2} 
only tells us that $\alpha$ has trivial kernel. Injectivity 
is not automatic, since $H^1(K, R_{E/F}(\bbG_m))$ and $H^1(K, G)$ 
are pointed sets with no group structure.}
Hence, we can identify $H^1(K, G)$ with
the kernel of $\delta$, and the lemma follows.
\end{proof}

\begin{remark} \label{rem.stably-rational}
When $n$ is odd, Lemma~\ref{lem.Galois-cohomology} tells us that
$H^1(K, G) = \{ 1 \}$ for every field $K/F$. In other words,
$G(E/F)$ is a special group.  Using the short exact 
sequence~\eqref{e.ses} one readily checks that $G(E/F)$ is rational 
over $F$. By Remark~\ref{rem.special}, we conclude that the Noether 
problem (SR) for this group has a positive solution. In other words,
every $F$-form of $\moduli{0}{n}$ is stably rational over $F$.  
This is a bit weaker than Theorem~\ref{thm.main}(a), which will be proved
in the next section.
\end{remark}

\section{Proof of Theorem~\ref{thm.main}(a)}
\label{sect.main-a}

Suppose $n = 2s + 1 \geqslant 5$ is odd. Our goal is to show that 
${ }^{\tau} \moduli{0}{n}$ is rational over $F$ for every infinite field $F$
and every $\tau \in H^1(F, \Sym_n)$. Let $E/F$ be the \'etale algebra
representing $\tau$. In view of Proposition~\ref{prop.Noether}(b), it suffices to show that
Noether's problem (R) for the representation ${ }^\tau \phi$ 
of the group $G(E/F)$ has a positive solution.

Recall that ${ }^\tau \phi$ is the natural  
representation of $G(E/F)$ on $F^2 \otimes_F E$ of $G(E/F)$.
The quotient $\mathbb A( F^2 \otimes_F E)/  \GL_2$ is
the Grassmannian $\Gr(2, E)$ (up to birational equivalence).
Thus the quotient   $\mathbb A( F^2 \otimes_F E)/ G(E/F)$ 
is birational to the quotient $\Gr(2, E)/ R_{E/F}(\bbG_m)$. 

Note that the diagonal subgroup 
$\bbG_m \hookrightarrow \bbG_m^n$ acts trivially on $\Gr(2, n)$. 
Hence, 
\[ \bbG_m = { }^\alpha (\bbG_m) \hookrightarrow { }^\alpha(\bbG_m^n) = 
R_{E/F}(\bbG_m) \]
acts trivially on $\Gr(2, E)$, and $R_{E/F}^0(\bbG_m):= R_{E/F}(\bbG_m)/\bbG_m$ acts
faithfully on $\Gr(2, E)$.
Our proof of the rationality of the quotient variety $\Gr(2, E)/R_{E/F}^0(\bbG_m)$ 
below is inspired by the arguments in~\cite{florence2}.

Fix an $F$-vector subspace $W$ of $E$ of dimension $s$, and define the rational map
\[ \begin{array}{r} f_W \colon \Gr(2, E) \dasharrow \check{\bbP}(E) \\ 
V  \to  V \cdot W, 
\end{array}  \]
where $V \cdot W$ is the $F$-linear span of elements of the form
$v \cdot w$ in $E$, as $v$ ranges over $V$ and $w$ ranges over $W$. Here
$v \cdot w$ stands for the product of $v$ and $w$ in $E$, and 
$\check{\bbP}(E)$ denotes the dual projective space to $\bbP(E)$. 
In other words, points of $\check{\bbP}(E)$ are 
$2s$-dimensional $F$-linear subspaces of $E$.

\begin{lem} \label{lem.well-defined}
(a) The dual projective space $\check{\bbP}(E)_0$ has a point $H$
whose orbit with respect to the natural action of $R_{E/F}^0(\bbG_m)$ 
is dense and whose stabilizer is trivial.

(b) Suppose $W \in \Gr(s, E)$ is such that
$f_W$ is well defined (i.e., 
$\dim(V \cdot W) = 2s$ for general $V \in \Gr(2, E)$).
Then $f_W$ is equivariant with respect to the natural 
action of $R_{E/F}^0(\bbG_m)$ on $\Gr(2, E)$ and $\check{\bbP}(E)$.

(c) There exists $W \in \Gr(s, E)$ defined over $F$ such that
$f_W$ is well defined and dominant.
\end{lem}

\begin{proof}
The assertions of parts (a) and (b) can be checked after passing to the separable closure
of $F^{sep}$ of $F$. In other words, we may assume that $F = F^{sep}$. In this case $E$ is the split algebra $F^n$,
$R_{E/F}^0(\bbG_m) = \bbG_m^{n}/\bbG_m$, and $\check{\bbP}(E) = \check{\bbP}^{n-1}$. 

(a) $(t_1, \dots, t_n) \in \bbG_m^n/\bbG_m$
takes the hyperplane $H \in \check{\bbP}(E)$ given by $c_1 x_1 + \dots + c_n x_n = 0$
to the hyperplane given by $(t_1^{-1} c_1) x_1 + \dots + (t_n^{-1}c_n) x_n = 0$. Thus any $H$
with $c_1, \dots, c_n \neq 0$ has a dense orbit in $\check{\bbP}(E)$ with trivial stabilizer.
In fact, all such $H$ lie in the same dense orbit; for future reference, we will denote this
dense orbit by $\check{\bbP}(E)_0$.

(b) Given $t = (t_1, \dots, t_n) \in \bbG_m^n$, we see that
\[ (tv) \cdot w = (t_1 a_1 b_1, \dots, t_n a_nb_n) = t (v \cdot w) \, . \]
for any $v = (a_1, \dots, a_n) \in V$ and $w = (b_1, \dots, b_n) \in W$. 
Hence, $(tV) \cdot W = t(V \cdot W)$, as desired. 

(c) Recall that the eigenvalues of $a \in E$ are the eigenvalues of the multiplication map $E \to E$ 
given by $x \mapsto ax$. They are elements of $F^{sep}$. 
Under an isomorphism
between $E \otimes_F F^{sep}$ and $(F^{sep})^n$ (over $F^{sep}$),
$a$ will be identified with an element of $(F^{sep})^n$ of the form
$(\lambda_1, \dots, \lambda_n)$, where $\lambda_1, \dots, \lambda_n$ are the eigenvalues of $a$.

Choose $a \in E$ with distinct eigenvalues in $F^{sep}$.
Elements of $E$ with distinct eigenvalues form a Zariski open subvariety $U$ of $\bbA(E)$
defined over $F$. Passing to $F^{sep}$, we see that $U \neq \emptyset$.
Since $F$ is assumed to be infinite, $F$-points are dense in $U$.
We choose $a$ to be one of these $F$-points,
and set $W = \Span_F(1, a, \dots, a^{s-1})$. We claim that for this choice of $W$,
the rational map $f_W$ is well defined and dominant. 

First let us show that $f_W$ is well defined. From the definition of $V \cdot W$ it is clear that
$\dim(V \cdot W) \leqslant 2s$ for any $V \in \Gr(2, E)$ and that equality holds 
for $V$ in a Zariski open subset of $\Gr(2, E)$. Thus in order to show that $f_W$ is a well-defined rational map,
it suffices to exhibit one element $V \in \Gr(2, E)$ such that $\dim(V \cdot W) = 2s$. 
We claim that $V = \Span_F(1, a^s)$ has this property, i.e.,
\[ V \cdot W = \Span_F(1, a, \dots, a^{s-1}, a^s, \dots, a^{2s-1})  \]
is a $2s$-dimensional subspace of $E$. It suffices to show that $1, a, \ldots, a^{2s}$ are linearly independent over $F$.
Passing to $F^{sep}$, we can write $a = (\lambda_1, \dots, \lambda_{2s+1})$, where $\lambda_1, \dots, \lambda_{2s+1}$
are distinct elements of $F^{sep}$. (Recall that $n = 2s+1$ throughout.) Since the $(2s + 1) \times (2s + 1)$ Vandermonde matrix
\[ \begin{pmatrix} 1 & \hdots & 1 \\ \lambda_1 & \hdots & \lambda_{2s + 1} \\ \hdotsfor{3} \\ \lambda_1^{2s} & \dots & \lambda_{2s + 1}^{2s}
\end{pmatrix} \]
is non-singular, we conclude that $1, a, \dots, a^{2s}$ are linearly independent over $F^{sep}$
and hence, over $F$, as desired. This shows that $f_W$ is well defined.

It remains to show that $f_W$ is dominant. By part (b), the image of $f_W$ is an $R_{E/G}^0(\bbG_m)$-invariant
subvariety of $\check{\bbP}(E)$. In view of part (a), it suffices to show that this subvariety intersects the dense
open orbit $\check{\bbP}(E)_0$. In fact, it suffices to show that $V \cdot W \in \check{\bbP}(E)_0$ for 
$V = \Span_F(1, a^s)$, as above. To do this, we may pass to $F^{sep}$ and thus identify $E \otimes_F F^{sep}$
with $(F^{sep})^n$. Then $a = (\lambda_1, \dots, \lambda_n)$,
where $\lambda_1, \dots, \lambda_n$ are distinct non-zero elements of $F^{sep}$. Recall from part (a)
that the complement of $\check{\bbP}(E)$ consists of hyperplanes of the form $c_1 x_1 + \dots + c_{2s} x_{2s} = 0$,
where $c_i = 0$ for some $i$ but $(c_1, \ldots, c_{2s}) \neq (0, \ldots, 0)$. It remains to show that 
the hyperplane $V \cdot W = \Span_F(1, a, \dots, a^{s-1}, a^s, \dots, a^{2s-1})$
is not of this form. Indeed, assume the contrary. By symmetry we may assume that the equation of the
hyperplane $\Span(1, \dots, a^{2s-1})$ in $E$ is
$c_1x_1 + \dots + c_{2s}x_{2s} = 0$, with $c_{2s+1} = 0$. Since $a^i \in V \cdot W$, this means that
\[ \text{$c_1 \lambda_1^i + \dots + c_{2s} \lambda_{2s}^i = 0$ for $i = 0, 1, \dots, 2s-1$.} \]
Since $\lambda_1, \dots, \lambda_{2s}$ are distinct, the $2s \times 2s$ Vandermonde matrix
\[ \begin{pmatrix} 1 & \hdots & 1 \\ \lambda_1 & \hdots & \lambda_{2s} \\ \hdotsfor{3} \\ \lambda_1^{2s-1} & \dots & \lambda_{2s}^{2s-1}
\end{pmatrix} \] is non-singular. This implies that $c_1 = \dots = c_{2s} = 0$, a contradiction. 
We conclude that $V \cdot W \in \check{\bbP}(E)_0$, as desired.
\end{proof}

We are now ready to finish the proof of Theorem~\ref{thm.main}(a).

Let $W \in \Gr(s, E)$ be the $s$-dimensional $F$-vector subspace of $E$ given 
by Lemma~\ref{lem.well-defined}. Choose 
a dense open $R_{E/F}^0(\bbG_m)$-invariant subvariety $U \subset \Gr(2, E)$ defined over $F$ such that
$f_W \colon \Gr(2, E) \dasharrow \check{\bbP}(E)_0$ restricts to a regular map on $U$, and the rational quotient map
$\Gr(2, E) \dasharrow \Gr(2, E)/R_{E/F}^0(\bbG_m)$ restricts to a $R_{E/F}^0(\bbG_m)$-torsor 
$\pi \colon U \to \Gr(2, E)/R_{E/F}^0(\bbG_m)$ (over a suitably chosen birational model of
$\Gr(2, E)/R_{E/F}^0(\bbG_m)$).
In summary, we obtain the following 
diagram of $R_{E/F}^0(\bbG_m)$-equivariant dominant rational maps:
\[ \xymatrix{  \Gr(2, E)  \ar@{-->}[r]^{f_W} & \check{\bbP}(E) &   \\
     U   \ar@{->}[d]_{\pi}^{\text{$R_{E/F}^0(\bbG_m)$-torsor}}  \ar@{^{(}->}[u]^{\text{open}} \ar@{->}[r]^{f_W} & \ar@{^{(}->}[u]_{\text{open}} \check{\bbP}(E)_0 \ar@{=}[r]^{\sim \quad } & R_{E/F}^0(\bbG_m) \\
 \quad \quad \quad \Gr(2, E)/R_{E/F}^0(\bbG_m). & &   } \]
Now choose an $F$-point $H \in \check{\bbP}(E)_0$; this can be done because we are assuming that $F$ is an infinite field.
  From the diagram, we see that  $f_W^{-1}(H) \subset U$ is  
 a section of $\pi$. In particular, $\Gr(2, E)/R_{E/F}^0(\bbG_m)$ is birationally  isomorphic to 
$f_W^{-1}(H)$ over $F$. It thus remains to show that $f_W^{-1}(H)$ is rational over $F$.

Let $Z$ be the $F$-vector subspace of $E$ given by $Z = \{ a \in E \, | \, a \cdot W \subset H \}$.  
Clearly $V \in U$ belongs to $\phi^{-1}(H)$ if and only if
$V \subset Z$ or equivalently, $V \in \Gr(2, Z)$. 
Thus $f_W^{-1}(H) = \Gr(2, Z) \cap U$ is a dense open subset of $\Gr(2, Z)$.
Clearly $f_W^{-1}(H)$ is non-empty. Since $\Gr(2, Z)$ is rational over $F$,
we conclude that $f_W^{-1}(H)$ is also rational over $F$, as desired.
\qed

\section{Proof of Theorem~\ref{thm.main}(b)}

We will deduce Theorem~\ref{thm.main}(b) from the following proposition.

\begin{prop} \label{prop.not-R-trivial}
Suppose $F$ is a field of characteristic $\neq 2$ and
$A = (a_1, a_2)$ is a quaternion division algebra over $F$,
for some $a_1, a_2 \in F^*$. Set $a_3 = a_1 a_2$ and
$E_i = F(\sqrt{a_i})$, for $i = 1, 2, 3$.
Consider the \'etale $F$-algebra 
\[ E = E_1^{n_1} \times E_2^{n_2} \times E_3^{n_3} \, , \] 
for some $n_1, n_2, n_3 \geqslant 1$.
Then Noether's problem (RR) has a negative solution for the group
$G(E/ F) = (\GL_2 \times R_{E/F}(\bbG_m))/\bbG_m$.
\end{prop}

Assuming that Proposition~\ref{prop.not-R-trivial} is established,
we can complete the proof of Theorem~\ref{thm.main} as follows. 
By a theorem of A.~S.~Merkurjev~\cite{merkurjev}, 
$\Br_2(F)$ is generated, as an abelian 
group by classes of quaternion algebras. Since we are assuming that
$\Br_2(F) \neq 0$, one of these classes, say, $(a_1, a_2)$ is non-split.
That is, $(a_1, a_2)$ is a division algebra. Since we are assuming that $n \geqslant 6$ is
even, we can choose $n_1, n_2, n_3 \geqslant 1$ so that
$n_1 + n_2 + n_3 = n$. For example, we can take
$n_1 = \dfrac{n}{2} - 2$, $n_2 = 1$ and $n_3 = 1$.
By Proposition~\ref{prop.not-R-trivial}, 
Noether's problem (RR) has a negative solution for the group
$G(E/ F) = (\GL_2 \times R_{E/F}(\bbG_m))/\bbG_m$. By 
Proposition~\ref{prop.Noether}(d), the $F$-form  
${ }^{\tau} \moduli{0}{n}$ of $\moduli{0}{n}$ is not retract rational
over $F$, where $\tau \in H^1(K, \Sym_n)$ is the class of the \'etale 
algebra $E/F$. This completes the proof of Theorem~\ref{thm.main}(b).
\qed

\begin{proof}[Proof of Proposition~\ref{prop.not-R-trivial}]
 Since $E_1$, $E_2$, $E_3$ are maximal subfields of $A$,
\[ A \otimes_F E_i \simeq \Mat_2(E_i) \]
for $i= 1, 2, 3$. In other words, $A$ is split by $E/F$.  
Thus by Lemma~\ref{lem.Galois-cohomology}, 
$A$ corresponds to a class in $H^1(F, G)$,
where $G := G(E/F)$. Denote this class by $\alpha$. 

Our assumption that there exists a non-split quaternion algebra
over $F$, implies that $F$ is an infinite field; see Remark (3) in the Introduction.
Thus Lemma~\ref{lem.R-connected} applies: it suffices to show that $\alpha$ 
is not $r$-trivial.  Assume the contrary.
Using Lemma~\ref{lem.Galois-cohomology} once again, we see that
this means the following:
there exists a quaternion algebra $A(t)$ over $F(t)$ such that 

\smallskip
(a) $A(t)$ is split by $F(t) \otimes_F E$, and

\smallskip
(b) $A(t)$ is unramified at $t = 0$ and $t = 1$, 
$A(0)$ is split over $F$, and $A(1)$ is isomorphic to $A$.

\smallskip
\noindent
Here $A(0)$ and $A(1)$ denote $A(t)$ specialized to the points
$t = 0$ and $t = 1$.  We now recall the Faddeev exact sequence
\begin{equation} \label{e.faddeev}
0 \to \Br(F) \to \Br(F(t)) \to \oplus_{\eta \in \bbP_F^1}  
H^1(F_{\eta}, \bbQ/\bbZ) \, ; 
\end{equation}
see e.g.,~\cite[Corollary 6.4.6]{gs}.
For $\eta \in \bbP^1_F$ denote the image 
of the Brauer class $[A(t)] \in \Br(F(t))$ in 
$H^1(F_{\eta}, \bbQ/\bbZ)$ by $\alpha_{\eta}$.  

By property (a) above, $A(t)$ is split 
by $E_i(t) := F(t) \otimes_F E_i = F(t)(\sqrt{a_i})$ 
for $i = 1, 2, 3$.  Note that $E_i(t)$ is a field
extension of $F(t)$ of degree $2$.  Since $A(t)$ is 
a quaternion algebra over $F(t)$, $A(t)^{\otimes 2}$ 
is split over $F(t)$ and hence, $2 \alpha_{\eta} = 0$ for every $\eta \in \bbP^1$. 
In particular, every $\alpha_{\eta}$ lies in 
$H^1(F_{\eta}, \bbZ/ 2 \bbZ) \hookrightarrow 
H^1(F_{\eta}, \bbQ/ \bbZ)$.

We claim that $\alpha_{\eta}$ is the trivial class in 
$H^1(F_{\eta}, \bbZ/ 2 \bbZ) = F_{\eta}^*/(F_{\eta}^*)^2$
for every $\eta \in \mathbb P^1$.
If we can prove this claim, then the
Faddeev exact sequence~\eqref{e.faddeev} will tell us that
$A(t)$ is constant, i.e., that $A(t)$ is isomorphic to
$B \otimes_F F(t)$ over $F(t)$, for some
quaternion algebra $B$ defined over $F$.
Consequently, $A(0)$ and $A(1)$ are both isomorphic to $B$
over $F$ and hence, are isomorphic to each other. 
Since $A(0)$ is split over $F$, and
$A(1) \simeq A$ is a quaternion division algebra,
this is a contradiction, and the proof of 
Proposition~\ref{prop.not-R-trivial} will be complete.

It remains to prove the claim. Assume the contrary. 
Suppose $\alpha_{\eta} = (b)$, where $(b)$ denotes 
the class of $b \in F_{\eta}^*$ in
$H^1(F_{\eta}, \bbZ/ 2 \bbZ) = F_{\eta}^*/(F_{\eta}^*)^2$. 
Since we are assuming $\alpha_{\eta} \neq (0)$, $b$ 
is not a square in $F_{\eta}^*$. 
On the other hand, since $A(t)$ splits over $F(t)(\sqrt{a_i})$, 
$b$ becomes a square in $F_{\eta}(\sqrt{a_i})^*$ for $i = 1, 2, 3$.
This is only possible if
$F_{\eta}(\sqrt{a_i})$ is a field extension of $F_{\eta}$
of degree $2$ and 
$\sqrt{b} = f_i \sqrt{a_i}$ for some $f_i \in F_{\eta}^*$, where
$i = 1$, $2$, $3$. Equivalently,
$b = f_i^2 a_i$ or $(b) = (a_i)$ in $H^1(F_{\eta}, \bbZ/p \bbZ)$.
Since $a_1 a_2 a_3$ is a complete square in $F^*$, we conclude that
\[ \alpha_{\eta} = (b) =
(b) + (b) + (b) = (a_1) + (a_2) + (a_3) = (a_1 a_2 a_3) = 0  \]
is the trivial class in 
$H^1(F_{\eta}, \bbZ/ 2 \bbZ) = F_{\eta}^*/(F_{\eta}^*)^2$,
a contradiction. This completes the proof of the claim and thus of 
Proposition~\ref{prop.not-R-trivial} and of Theorem~\ref{thm.main}(b).
\end{proof}

\section*{ Acknowledgements}
We are grateful to J.-L.~Colliot-Th\'el\`ene, A. Massarenti,
M.~Mella, M.~Talpo and A. Vistoli for stimulating conversations.
We also thank the anonymous referee for helpful comments.


\end{document}